\documentclass[12pt]{article}
\usepackage{a4wide}

\usepackage[a4paper, margin=2.3cm]{geometry}
\usepackage{amsmath,amssymb,amsthm}
\usepackage{color}
\usepackage{parskip}
\usepackage{hyperref}
\usepackage{parskip}
\usepackage{setspace}
\usepackage{graphicx}

\newtheorem{theorem}{Theorem}[section]
\newtheorem{construction}[theorem]{Construction}
\newtheorem{remark}{Remark}[section]
\newtheorem{corollary}[theorem]{Corollary}
\newtheorem{lemma}[theorem]{Lemma}

\newtheorem{conjecture}[theorem]{Conjecture}
\newtheorem*{definition*}{Definition}

\begin{document}

\title{Product of sets on varieties in finite fields}
\author{Che-Jui Chang\and Ali Mohammadi\and Thang Pham \and Chun-Yen Shen}
\maketitle

\begin{abstract} 
Let $V$ be a variety in $\mathbb{F}_q^d$ and $E\subset V$. It is known that if any line passing through the origin contains a bounded number of points from $E$, then $|\prod(E)|=|\{x\cdot y\colon x, y\in E\}|\gg q$ whenever $|E|\gg q^{\frac{d}{2}}$. In this paper, we show that the barrier $\frac{d}{2}$ can be broken when $V$ is a paraboloid in some specific dimensions. The main novelty in our approach is to link this question to the distance problem in one lower dimensional vector space, allowing us to use recent developments in this area to obtain improvements.
\end{abstract}



\maketitle 
\section{Introduction}
Let $\mathbb{F}_q$ be a finite field of order $q$, where $q$ is a prime power. For $E, F\subset \mathbb{F}_q^d$, the set of dot products between $E$ and $F$ is defined by 
\[\prod(E, F):=\{x\cdot y\colon x\in E, y\in F\}\subset\mathbb{F}_q.\]
When $E=F$, we write $\prod(E)$ instead of $\prod(E, F)$. In \cite{hart}, Hart, Iosevich, Koh, and Rudnev studied the question of finding the smallest exponent $\alpha$ such that if $|E||F|\gg q^{\alpha}$, then $|\prod(E, F)|\gg q$. Here and throughout the paper, we use the notation $X\gg Y$ if there exists an absolute constant $c>0$ such that $X\ge cY$. 

By using discrete Fourier analysis, they proved the following result. 
\begin{theorem}[Hart-Iosevich-Koh-Rudnev, \cite{hart}]
Let $E$ be a set in $\mathbb{F}_q^d$. Suppose that $|E|>q^{\frac{d+1}{2}}$, then 
\[\mathbb{F}_q\setminus \{0\}\subset \prod(E).\]
Moreover, this result is sharp in the following sense: 
\begin{enumerate}
    \item If $\mathbb{F}_q$ is a quadratic extension, for any $\epsilon>0$, there exists $E\subset \mathbb{F}_q^d$ of size $q^{\frac{d+1}{2}-\epsilon}$ such that $|\prod(E)|=o(q)$. 
    \item If $d\equiv 3\mod 4$ and $q$ is large enough, then for any $t\ne 0$, there exists $E\subset \mathbb{F}_q^d$ of size about $q^{\frac{d+1}{2}}$ such that $t\not\in \prod(E)$.
\end{enumerate}
\end{theorem}

It is natural to ask under what additional conditions, the exponent $\frac{d+1}{2}$ can be improved if we are only interested in a positive proportion of all elements in the field. In the same paper, Hart et al. showed that when $E$ is a subset of the unit sphere, then the exponent $\frac{d}{2}$ is enough. This result can be extended for general sets $E\subset \mathbb{F}_q^d$ whenever $E$ does not contain many points on any lines through the origin. We refer the reader to \cite[Section 3.1]{hart} and \cite[Theorem 1.3]{PV} for more details and discussions. To the best of our knowledge, no improvement of $d/2$ has been made in the literature for spheres or other varieties. 

In this paper, we are interested in finding varieties $V$ for which the threshold $\frac{d}{2}$ can be improved. It follows from our main theorem (Theorem \ref{th:main}) that paraboloids in some specific dimensions provide the first model for this type question. The main novelty in our approach is to link this question to the distance problem in one lower dimensional vector space, allowing us to use recent developments in this area to obtain improvements. To state our main theorems, we need to recall some notations from Fourier restriction theory. 

Let $(\mathbb{F}_q^d, dx)$ be the $d$-dimensional vector space over $\mathbb{F}_q$ endowed with the normalized counting measure $dx$, and $(\mathbb{F}_q^d, dc)$ be the dual space with the counting measure $dc$. For complex-valued functions $f\colon (\mathbb{F}_q^d, dx)\to \mathbb{C}$ and $g\colon (\mathbb{F}_q^d, dc)\to \mathbb{C}$, we define 
\[\int f(x)dx:=q^{-d}\sum_{x\in \mathbb{F}_q^d}f(x),\quad ~\int g(c)dc:=\sum_{c\in \mathbb{F}_q^d}g(c).\]
Let $V$ be an algebraic variety in $(\mathbb{F}_q^d, dx)$, we define the normalized surface measure $d\sigma$ on $V$ by 
\[d\sigma(x):=q^d|V|^{-1}1_V(x)dx.\]
So, for any function $f\colon V\to \mathbb{C}$, 
\[\int f(x)d\sigma(x):=|V|^{-1}\sum_{x\in V}f(x).\]
For a function $f\colon (\mathbb{F}_q^d, dx)\to \mathbb{C}$, the Fourier transform $\widehat{f}$ is defined on the space $(\mathbb{F}_q^d, dc)$ by
\[\widehat{f}(c):=\int \chi(-x\cdot c)f(x)dx=q^{d}\sum_{x\in \mathbb{F}_q^d}\chi(-x\cdot c)f(x), ~c\in (\mathbb{F}_q^d, dc).\]
Similarly, for a function $g\colon (\mathbb{F}_q^d, dc)\to \mathbb{C}$, its Fourier transform is defined on the space $(\mathbb{F}_q^d, dx)$ by
\[\widehat{g}(x):=\int \chi(-x\cdot c)g(c)dc=\sum_{c\in \mathbb{F}_q^d}\chi(-x\cdot c)g(c).\]

With the normalized surface measure $d\sigma$ on $V$ and a function $f\colon (\mathbb{F}_q^d, dx)\to \mathbb{C}$, we define the inverse Fourier transform $(fd\sigma)^\vee$ of the measure $fd\sigma$ by 
\[(fd\sigma)^\vee(c):=\int \chi(c\cdot x)fd\sigma(x)=|V|^{-1}\sum_{x\in V}\chi(c\cdot x)f(x),\]
for $c\in (\mathbb{F}_q^d, dc)$.

The $L^2 \to L^r$ extension problem for the variety $V$ is to determine  all ranges of $r$ such that the following inequality 
\begin{equation}\label{defR}||fd\sigma^\vee||_{L^{r}(\mathbb{F}_q^d, dc)}\le C ||f||_{L^2(V, d\sigma)}\end{equation}
holds for any function $f$ on $V$. We note that in the above inequality, the constant $C$ is independent of $q$ (the size of $\mathbb{F}_q$). There is a series of papers studying $L^2\to L^r$ estimates for various varieties in the literature, for instance, see \cite{IK, Lew, MT, RS18} and the references therein. In this paper, we require estimates associated to spheres of non-zero radius. 

For a positive integer $d\ge 3$ and a non-zero element $r\in \mathbb{F}_q$, the paraboloid $P_d$ and the sphere $S_r$ centered at origin of radius $r$ in $\mathbb{F}_q^d$ are defined by the following formulas:
\[P_d:=\left\lbrace x=(x_1, \ldots, x_d)\colon x_d=x_1^2+\cdots+x_{d-1}^2 \right\rbrace,\]
and 
\[S_r:=\left\lbrace x=(x_1, \ldots, x_d)\colon x_1^2+\cdots+x_d^2=r\right\rbrace.\]
Our main result is as follows. 
\begin{theorem}\label{th:main}
Let $E$ be a set in $P_d$ with $d\equiv 3\mod 4$ and $q\equiv 3\mod 4$. Assume that the extension conjecture 
\[||fd\sigma^\vee||_{L^\frac{2d+2}{d-1}(\mathbb{F}_q^{d-1}, dc)}\ll ||f||_{L^2(S_r, d\sigma)},\]
holds for any $S_r\subset \mathbb{F}_q^{d-1}$ and $r\ne 0$, then we have 
\[|\prod(E)|\gg q,\]
whenever $|E|\gg q^{\frac{(d-1)^2+2(d-1)}{2(d-1)+2}}=q^{\frac{d}{2}-\frac{(d+1)}{2(d-1)+2}}$.
\end{theorem}

It is worth noting that the same conclusion does not hold when $d$ is even. When $d\equiv 3\mod 4$ and $q\equiv 3\mod 4$, we conjecture that the sharp exponent should be $(d-1)/2$. To support these claims, we provide constructions in the last section.

\begin{corollary}\label{cor1.3}
Let $E$ be a set in $P_3 \subset \mathbb{F}_q^3$ with $q\equiv 3\mod 4$. Suppose that $|E|\gg q^{\frac{3}{2}-\frac{1}{6}}$, then
\[|\prod(E)|\gg q.\]
\end{corollary}

If we assume $q$ is an odd prime number, then by using a recent theorem on \textit{bisector line energy} due to Murphy, Petridis, Pham, Rudnev, and Stevens \cite{MPPRP}, we can get a better exponent, namely, $\frac{5}{4}$ instead of $\frac{4}{3}$. 

\begin{theorem}\label{th:primefields}
Let $\mathbb{F}_p$ be a prime field, and $E$ be a set in $P_3$ in $\mathbb{F}_p^3$ with $p\equiv 3\mod 4$. Suppose that $|E|\gg p^{\frac{3}{2}-\frac{1}{4}}$, then 
\[|\prod(E)|\gg p.\]
Moreover, if $|E|\ll p^{5/4}$ and $|E\setminus \{(x_1, x_2, 0)\colon (x_1, x_2)\in \mathbb{F}_p^2\}|\gg |E|$, then we also have 
\[|\prod(E)|\gg |E|^{\frac{2}{3}}.\]
\end{theorem}

It is not clear to us how the method of this paper can be adapted for other varieties, say spheres, we hope to address this question in a sequel paper. We also note that for spheres, the dot product set $\prod(E)$ is of the same size as the distance set $\Delta(E)$, where $\Delta(E):=\{||x-y||\colon x, y\in E\}$. The exponent $d/2$ has been obtained in \cite[Theorem 2.8]{hart}.




\section{Preliminary: Extension estimates}
As mentioned in the introduction, the $L^2 \to L^r$ extension problem for the variety $V$ is to determine  all ranges of $r$ such that the following inequality 
\begin{equation}\label{defR}||fd\sigma^\vee||_{L^{r}(\mathbb{F}_q^d, dc)}\le C ||f||_{L^2(V, d\sigma)}\end{equation}
holds for any function $f$ on $V$.

In this paper, we only need extension results for spheres. The following is the well-known $L^2\to L^r$ extension conjecture in $\mathbb{F}_q^n$. We refer the reader to \cite{KPV} for more discussions.
\begin{conjecture}\label{concon}
For even $n\ge 2$, let $S_r$ be the sphere centered at the origin of radius $r$ with $r\ne 0$ in $\mathbb{F}_q^n$. We have the following $L^2\to L^r$ extension estimate 
\[R_{S_r}^*\left(2\to \frac{2n+4}{n}\right)\ll 1.\]
\end{conjecture}

It was proved in \cite{chapman} that this conjecture is true for $n=2$, namely,

\begin{theorem}
Let $C_r$ be the circle centered at the origin of radius $r$ with $r\ne 0$ in $\mathbb{F}_q^2$. We have the following $L^2\to L^r$ extension estimate
\[R_{C_r}^*\left(2\to 4\right)\ll 1.\]
\end{theorem}
Although Conjecture \ref{concon} is still wide open in dimensions $n\ge 4$, for the sphere of radius $0$, denoted by $S_0$, it has been shown in \cite{IKPSL} that the same conclusion holds. 
\begin{theorem}\label{zerosphere}
Let $S_0$ be the sphere centered at the origin of radius $0$. Assume $n\equiv 2\mod 4$ and $q\equiv 3\mod 4$, then the following $L^2\to L^r$ extension estimate holds:
\[R_{S_0}^*\left(2\to \frac{2n+4}{n}\right)\ll 1.\]
\end{theorem}

With these results in hand, we are ready to prove Theorem \ref{th:main} in the next section. 
\section{Proof of Theorem \ref{th:main}}
The proof of Theorem \ref{th:main} contains two main steps: Reducing to the triangle problem and Bounding the number of isosceles triangles. 
\subsection{Reducing to the isosceles triangles problem}

By the Cauchy-Schwarz inequality, we observe that 
\begin{equation}\label{first-eq}\left|\prod(E)\right|\gg \frac{|E|^3}{|D(E)|},\end{equation}
where $D(E)$ is the number of triples $(x, y, z)\in E^3$ such that $x\cdot y=x\cdot z$. To see this, first, by Cauchy-Schwarz inequality, we have  
\begin{equation}\label{first-eq}\left|\prod(E)\right|\geq \frac{|E|^4}{|M(E)|},\end{equation} 
where $M(E)=\{(x, y, w, z) \in E^4: x\cdot y= w\cdot z\}.$ Thus it suffices to show $|M(E)| \leq |E| |D(E)|$.
Now for a given $t$ and $x\in E$, write $\pi_x^t(E)=\{y \in E, x\cdot y=t\}.$ Then, we observe that
$$|M(E)|=\sum_{t} \left(\sum_{x} |\pi_x^t(E)|\right)^2.$$ By Cauchy-Schwarz inequality, we have 
$$\sum_{t} \left(\sum_{x} |\pi_x^t(E)|\right)^2 \leq \sum_{t} |E| \sum_{x} |\{(y, z) \in E^2, x\cdot y=x\cdot z=t\}| = |E||D(E)|.$$
For any point $x=(x_1, \ldots, x_d)\in E\subset P_d$, we define $\overline{x}:=(x_1, \ldots, x_{d-1})$, and let $\overline{E}:=\{\overline{x}\colon x\in E\}\subset \mathbb{F}_q^{d-1}$. 

Under our assumptions on the set $E$, without loss of generality, we may assume that $||\overline{x}||\ne 0$ for all $x\in E$.

For $x, y, z\in P_d$, the identity $x\cdot y=x\cdot z$ can be rewritten as 
\[(\overline{x}, ||\overline{x}||)\cdot (\overline{y}-\overline{z}, ||\overline{y}||-||\overline{z}||)=0.\]
This implies that
\[\left(\frac{\overline{x}}{||\overline{x}||}, 1\right)\cdot \left(\overline{y}-\overline{z}, ||\overline{y}||-||\overline{z}||\right)=0.\]
So 
\begin{equation}\label{eq:triangle}
\left\lVert\frac{-\overline{x}}{2||\overline{x}||}-\overline{y}\right\rVert=\left\lVert\frac{-\overline{x}}{2||\overline{x}||}-\overline{z}\right\rVert.\end{equation}
Set $F':=\left\lbrace \frac{-\overline{x}}{2||\overline{x}||}\colon \overline{x}\in E'\right\rbrace\subset \mathbb{F}_q^{d-1}$. 

The equation (\ref{eq:triangle}) counts the number of isosceles triangles with one vertex from $F'$ and the two other vertices (base) from $E'$. 

In other words, to bound the size of $D(E)$ from above, it is enough to count the number of isosceles triangles with vertices in $F'$ and $E'$ satisfying the relation (\ref{eq:triangle}).

\subsection{Bounding the number of isosceles triangles}
Given $X\subset \mathbb{F}_q^d$ and $y\in \mathbb{F}_q^d$, we first count the number of isosceles triangles with a given apex. 
\begin{lemma}\label{lem2.3}
Let $X \subset \mathbb F_q^n$ and $y\in X.$ Then
we have
$$ \sum_{x,z\in X: ||x-y||=||z-y||\ne 0} 1 \ll \frac{|X|^2}{q}+  q^n \sum_{r\in \mathbb F_q^*} \left| \sum_{m\in S_r} \widehat{X}(m) \chi(y\cdot m)\right|^2+q^n\left\vert \sum_{||m||=0, m \ne 0}\widehat{X}(m)\chi(y\cdot m) \right\vert^2.$$ 
\end{lemma}
\begin{proof}

Let $O(n)$ be the orthogonal group of $n\times n$ matrices in $\mathbb{F}_q$. It is well-known that $|O(n)|=(1+o(1))q^{\frac{n^2}{2}}$, and the stabilizer of any non-zero vector in $\mathbb{F}_q^n$ is of the size $|O(n-1)|$. It is not hard to prove that $|O(n)|=|S_1||O(n-1)|=(1+o(1))q^{n-1}|O(n-1)|$. We note that $O(d)$ acts transitively on the set of non-zero vectors of any given norm. 

We first note that if there is an $x \in X$ with $y-x \in S_r$, the sphere with radius $r$ in $\mathbb{F}_q^n$, then writing $y-x = x'$, we have a one-to-one correspondence between $x' \in S_r$ and $y-x' \in X$. 
Hence, we can write   
\begin{align*} \sum_{x,z\in X: ||x-y||=||z-y||\ne 0} 1 &=\sum_{r\in \mathbb F_q^*} \sum_{x\in S_r} X(y-x) \sum_{z\in S_r} X(y-z)\\
&\le \frac{1}{|O(n-1)|} \cdot \sum_{\theta\in O(n)} \sum_{r\in \mathbb F_q^*} \sum_{x\in S_r} X(y-x) X(y-\theta{x}).
\end{align*}
Applying the Fourier inversion theorem to functions $X(y-x), X(y-\theta{x}),$ that is
$$X(y-x) =\sum_{m \in \mathbb F_q^n} \widehat{X}(m) \chi(m(y-x)).$$
We have
$$\sum_{x,z\in X: ||x-y||=||z-y||\ne 0} 1 \le \frac{1}{|O(n-1)|} \cdot \sum_{\theta\in O(n)} \sum_{m,m'\in \mathbb F_q^n} \widehat{X}(m) \widehat{X}(m') \chi(y\cdot (m+m'))  \sum_{x\in \mathbb F_q^n} \chi(-m\cdot x-m'\cdot \theta{x}).$$
By the orthogonality of $\chi$, we compute the above sum in $x\in \mathbb F_q^d$, then, the right-hand size of the above inequality becomes
$$\frac{q^n}{|O(n-1)|} \cdot\sum_{\theta\in O(n)} \sum_{m\in \mathbb F_q^n} \widehat{X}(m) \widehat{X}(-\theta{m}) \chi(y\cdot (m-\theta{m})),$$
which can be decomposed as the sum of 
$$\frac{q^n}{|O(n-1)|} \sum_{\theta\in O(n)} \sum_{m\in S_0} \widehat{X}(m) \widehat{X}(-\theta{m}) \chi(y\cdot (m-\theta{m}))$$  and $$\frac{q^n}{|O(n-1)|} \sum_{\theta\in O(n)} \sum_{r\in \mathbb F_q^*}\sum_{m\in S_r} \widehat{X}(m) \widehat{X}(-\theta{m}) \chi(y\cdot (m-\theta{m})),$$
which is equal to
\begin{align*}
&= \frac{q^n}{|O(n-1)|}\sum_{\theta\in O(n)} \sum_{m\in S_0} \widehat{X}(m) \widehat{X}(-\theta{m}) \chi(y\cdot (m-\theta{m})) + q^n \sum_{r\in \mathbb F_q^*} \left| \sum_{m\in S_r} \widehat{X}(m) \chi(y\cdot m)\right|^2\\
&\ll \frac{|X|^2}{q}+  q^n \sum_{r\in \mathbb F_q^*} \left| \sum_{m\in S_r} \widehat{X}(m) \chi(y\cdot m)\right|^2+q^n\left\vert \sum_{||m||=0, m \ne 0}\widehat{X}(m)\chi(y\cdot m) \right\vert^2.
 \end{align*}
\end{proof}

Lemma \ref{lem2.3} shows that the number of isosceles triangles can be reduced to extension-type estimates associated to spheres. Thus, we now can apply results in Section 2 to derive the next theorem.

\begin{theorem}\label{th:isos}
For $n\equiv 2\mod 4$ and $X\subset \mathbb{F}_q^n$ with $q\equiv 3\mod 4$. Assume that Conjecture \ref{concon} holds, then the number of isosceles triangles is bounded by 
\[\ll \frac{|X|^3}{q}+q^{d-1}|X|^{\frac{n+4}{n+2}}+q^{\frac{n-2}{2}}|X|^2.\]
\end{theorem}
\begin{proof}
Let $T^{\mathtt{nde}}(X)$ be the number of isosceles triangles in $E$ of the form $(x, y, z)\in X^3$ such that $||x-y||=||x-z||\ne 0$. Let $T^{\mathtt{de}}(X)$ be the number of triangles with at least one side of zero length. 

To bound $T^{\mathtt{de}}(X)$, we will show that the number of pairs $(x, y)\in X\times X$ such that $||x-y||=0$ is at most 
\[\frac{|X|^2}{q}+q^{\frac{n-2}{2}}|X|.\]
In other words, once we have the bound above. Then
$$  \sum_{x,y, z\in X: ||x-y||=||z-y|| =0} 1 \leq \sum_{z \in X }\sum_{x, y \in X: ||x-y||=0} 1 \leq \frac{|X|^3}{q} + q^{\frac{n-2}{2}}|X|^2.$$
Now write 
$$\sum_{x, y \in X: ||x-y||=0} 1= \sum_{x, y \in \mathbb F_q^n} X(x)X(y) S_0(x-y),$$
which, by the Fourier inversion formula, becomes
$$ \sum_{x, y \in \mathbb F_q^n} X(x)X(y) \sum_{m \in \mathbb F_q^n} \widehat{S_0}(m)\chi{((x-y)m)},$$
which is $\sum_{m  \in \mathbb F_q^n} |\widehat{X}(m)|^2 \widehat{S_0}(m).$
In order to proceed further, we recall the following lemma on the Fourier transform of the sphere of zero radius from \cite{IKPSL}.
\begin{lemma}[\cite{IKPSL}] \label{ExplicitS0}Let $S_0$ be the sphere with zero radius in $\mathbb F_q^n.$ Assume that $n=4k+2$ for $k\in \mathbb N$ and $q\equiv 3\mod 4.$
Then we have
$$ \widehat{S_0}(m) :=q^{-n} \sum_{y\in S_0} \chi(m\cdot y) = q^{-1} \delta_0(m) -q^{\frac{-(n+2)}{2}} \sum_{r\ne 0} \chi(r\| m \|),$$
where $\delta_0(m)=1 $ for $m=(0,\ldots,0)$, and $0$ otherwise.
\end{lemma} 

Now inserting the formula for $\widehat{S_0}(m)$, we get 
$$ \sum_{ m \in \mathbb F_q^n} |\widehat{X}(m)|^2 q^{-1} \delta_0(m) -q^{\frac{-(n+2)}{2}} \sum_{m\in \mathbb F_q^n} |\widehat{X}(m)|^2 \sum_{r\ne 0} \chi(r\|m\|).$$
Applying the orthogonality relation of $\chi$ to the sum over $r\ne 0$, we obtain
$$ |\widehat{X}(0,\ldots,0)|^2 q^{-1}  -q^{\frac{-(n+2)}{2}} (q-1) \sum_{\|m\|=0} |\widehat{X}(m)|^2  + q^{\frac{-(n+2)}{2}} \sum_{\|m\|\ne 0} |\widehat{X}(m)|^2$$
$$= q^{-1}|X|^2 -q^{\frac{-(n+2)}{2}} q\sum_{\|m\|=0} |\widehat{X}(m)|^2 + q^{\frac{-(n+2)}{2}} \sum_{m\in \mathbb F_q^n} |\widehat{X}(m)|^2$$
Since $\sum_{m\in \mathbb F_q^n} |\widehat{X}(m)|^2 =q^n |X|$ and the middle term above is negative, we get that
$$\sum_{m  \in \mathbb F_q^n} |\widehat{X}(m)|^2 \widehat{S_0}(m) \leq \frac{|X|^2}{q} + q^{\frac{n-2}{2}}|X|.$$

Hence, 
\[T^{\mathtt{de}}(X)\ll \frac{|X|^3}{q}+q^{\frac{n-2}{2}}|X|^2.\]
To bound $T^{\mathtt{nde}}$, we observe that
\[T^{\mathtt{nde}}(X)=\sum_{y\in X}\sum_{\substack{x, z\in X:\\||x-y||=||z-y||}}1.\]
Thus, applying Lemma \ref{lem2.3}, it suffices to bound the following sums:
\[\sum_{y\in X}\left\vert \sum_{m\in S_r}\widehat{X}(m)\chi(y\cdot m)\right\vert^2 ~\mbox{with} ~r\ne 0,\] and  \[\sum_{y\in X}\left\vert \sum_{m\in S_0}\widehat{X}(m)\chi(y\cdot m)\right\vert^2.\]
For the first sum,  
\begin{align*}
&\sum_{y\in X}\left\vert \sum_{m\in S_r}\widehat{X}(m)\chi(y\cdot m)\right\vert^2=|S_r|^2\sum_{y\in X}|fd\sigma^\vee(y)|^2\\
&\le |S_r|^2\cdot |X|^{\frac{2}{n+2}}\cdot ||fd\sigma^\vee||_{L^{\frac{2n+4}{n}}(\mathbb{F}_q^n, dc)}^2.
\end{align*}

Assuming Conjecture \ref{concon} holds, i.e.  
\[||fd\sigma^\vee||_{L^\frac{2n+4}{n}(\mathbb{F}_q^n, dc)}\ll ||f||_{L^2(S_r, d\sigma)},\]
we have
\begin{align*}
&\sum_{y\in X}\left\vert \sum_{m\in S_r}\widehat{X}(m)\chi(y\cdot m)\right\vert^2=|S_r|^2\sum_{y\in X}|fd\sigma^\vee(y)|^2\\
&\le |S_r|^2\cdot |X|^{\frac{2}{n+2}}\cdot ||fd\sigma^\vee||_{L^{\frac{2n+4}{n}}(\mathbb{F}_q^n, dc)}^2\\
&\le |S_r|^2\cdot |X|^{\frac{2}{n+2}}\cdot ||f||_{L^2(S_r, d\sigma)}^2, ~f=\widehat{X}\\
&=|S_r|\cdot |X|^{\frac{2}{n+2}}\sum_{m\in S_r}|\widehat{X}(m)|^2.
\end{align*}

Similarly, using Theorem \ref{zerosphere}, we have the same bound for the second sum. Using the fact that $|S_r|=(1+o(1))q^{n-1}$, we have
\begin{align*}
    &T^{\mathtt{nde}}(X)\ll \frac{|X|^3}{q}+q^{n-1}\sum_{r\in \mathbb{F}_q}|X|^{\frac{2}{n+2}}\sum_{m\in S_r}|\widehat{X}(m)|^2\\
    &=\frac{|X|^3}{q}+q^{n-1}|X|^{\frac{n+4}{n+2}}.
\end{align*}
Putting the bounds of $T^{\mathtt{nde}}(X)$ and $T^{\mathtt{de}}(X)$ together gives us the desired estimate. 
\end{proof}
\subsection{Concluding the proof}
Setting $X=E'\cup F'\subset \mathbb{F}_q^{d-1}$. We have $|X|\le 2|E|.$ It is not hard to see that $D(E)$ is bounded by the number of isosceles triangles in $X$. So applying Theorem \ref{th:isos} and (\ref{first-eq}) concludes the proof. 

\section{Proof of Theorem \ref{th:primefields}}
We follow the proof of Theorem \ref{th:main} identically, except that we have a more effective bound on the number of isosceles triangles in two dimensions due to Murphy, Petridis, Pham, Rudnev and Stevens \cite{MPPRP}. 

Given a set $X\subset \mathbb{F}_p^{2}$, we say that a triple $(x, y, z)\in X^3$ forms a \textit{non-degenerate} isosceles triangle if $||x-y||=||x-z||$ and $||y-z||\ne 0$. If $||x-y||=||x-z||$ and $||y-z||=0$, we say the triangle is \textit{degenerate}.

\begin{theorem}[Non-degenerate isosceles triangles]
Let $X$ be a set in $\mathbb{F}_p^2$ with $|X|\le p^{4/3}$. Let $T^*(X)$ be the number of non-degenerate isosceles triangles in $X$. We have 
\[T^*(X)-\frac{|X|^3}{p}\ll \min \left\lbrace p^{2/3}|X|^{5/3}+p^{1/4}|X|^2,~ |X|^{7/3}\right\rbrace\]
\end{theorem}
Hence,
\begin{enumerate}
    \item if $|X|\gg p^{5/4}$,  
\[T^*(X)\ll \frac{|X|^3}{p}.\]
\item if $|X|\ll p^{5/4}$, then
\[T^*(X)\ll |X|^{7/3}.\]
\end{enumerate}

Since we assumed that $p\equiv 3\mod 4$, the number of degenerate isosceles triangles is at most $\ll |E|^2$. Hence, 
\begin{enumerate}
    \item if $|E|\gg p^{5/4}$,  
\[|D(E)|\ll \frac{|E|^3}{p}+|E|^2\]
\item if $|E|\ll p^{5/4}$, then
\[|D(E)|\ll |E|^{7/3}.\]
\end{enumerate}
These give us the desired bounds of Theorem \ref{th:primefields}. The first construction tells us that it is impossible to break $\frac{d}{2}$ in even dimensions.
\section{Constructions and Remarks}
We have the following constructions on the sharpness of Theorem \ref{th:main}.
\begin{construction}\label{lem4.1}
Assume $d$ is even, for any $\epsilon>0$, there exists a set $E\subset P_d$ such that $|E|\sim q^{\frac{d}{2}-\epsilon}$ such that $|\prod(E)|=o(q)$. 
\end{construction}
\begin{proof}
We first consider the case $d\equiv 2\mod 4$. We know from Lemma 5.1 in \cite{hart} that there exist $\frac{d-2}{2}$ nonzero vectors $v_1, \ldots, v_{\frac{d-2}{2}}$ in $\mathbb{F}_q^{d-2}$ which are mutually orthogonal, i.e. $v_i\cdot v_j=0$ for all $1\le i\le j\le \frac{d-2}{2}$. Let $S$ be the subspace spanned by these $(d-2)/2$ vectors. Set $E=S\times \{(x, x^2)\colon x\in A\},$
where $A$ is a multiplicative subgroup of $\mathbb{F}_q^*$ of size $q^{1-\epsilon}$. Then one can directly check that 
\[\prod(E)\subset a+a^2,\] for $a \in A$. This shows that $|\prod(E)|\sim q^{1-\epsilon}$ and $|E| \sim q^{\frac{d}{2}-\epsilon}$. 

When $d\equiv 0\mod 4$, we use Lemma 5.1 from \cite{hart} again to obtain $\frac{d}{2}$ vectors that are mutually orthogonal in $\mathbb{F}_q^d$. We denote these vectors by $u_1,\ldots, u_{\frac{d}{2}}$. Let $A$ be a multiplicative subgroup of $\mathbb{F}_q^*$ of size $q^{1-\epsilon}$. We note that $v_{\frac{d}{2}}$ is of the form $(0, \ldots, 0, 1, i)$, where $i^2=-1$. Define
\[S:=\mathbb{F}_qv_1+\cdots+\mathbb{F}_qv_{\frac{d}{2}-1}+Av_{\frac{d}{2}}.\]
Set 
\[E=\{(x_1, \ldots, x_{d-1}, - x_d^2)\colon (x_1, \ldots, x_d)\in S\}.\]
Since $|S|\sim q^{\frac{d}{2}-\epsilon}$, we have $|E|\sim q^{\frac{d}{2}-\epsilon}$. 

For $(x_1, \ldots, x_{d-1}, -x_d^2)$ and $(y_1, \ldots, y_{d-1}, -y_d^2)$ in $E$, we have their product is 
\[x_1y_1+\cdots+x_{d-1}y_{d-1}-x_d^2y_d^2=-x_dy_d-x_d^2y_d^2.\]
So the product value becomes $x+x^2$ for $x\in A$. This implies $|\prod(E)| \sim q^{1-\epsilon}$.
\end{proof}
The next construction provides the information that the best exponent of Theorem \ref{th:main} one can expect is $\frac{d-1}{2}$. 
\begin{construction}
Assume $d\equiv 3\mod 4$  and $q\equiv 3\mod 4$, for any $\epsilon>0$, there exists a set $E\subset P_d$ such that $|E|\sim q^{\frac{d-1}{2}-\epsilon}$ such that $|\prod(E)|=o(q)$. 
\end{construction}
\begin{proof}
Following the first case of Construction \ref{lem4.1}, we may find a subspace $S^{'} \subset \mathbb{F}_q^{d-3}$ of size $q^{\frac{d-3}{2}}$, with the property that any pair of its vectors are mutually orthogonal. Let $S\subset\mathbb{F}_q^{d-2}$ be the set one gets by adjoining $0$ as the last entry of elements of $S^{'}$. Then, by choosing $E$ in the same way as the first case of Construction \ref{lem4.1}, we get $|\prod(E)| \sim q^{1-\epsilon}$ while $|E|\sim q^{\frac{d-1}{2}-\epsilon}$
\end{proof}
\begin{remark}
It is well-known that the $L^2$-norm of the distance problem, i.e. the number of quadruples $(x, y, z, w)\in E^4$ such that $||x-y||=||z-w||$, can be bounded by using extension estimates, see \cite[Theorem 1.7]{KPV} for example. By using the Cauchy-Schwarz inequality, the number of such quadruples is at most $|E|$ times the number of isosceles triangles in $E$. In other words, Theorem \ref{th:isos}, provided in Section $3$, offers a stronger form of this problem.
\end{remark}
\begin{remark}
In the statement of Corollary \ref{cor1.3}, if  $q\equiv 1\mod 4$ then we find that the exponent $\frac{4}{3}$ is not good enough to guarantee that the number of isosceles triangles (including both degenerate and non-degenerate) is at most $\ll |E|^3/q$. This can be seen by taking $E$ to be a set of points on $|E|/M$ parallel lines of slope $i$, where each line contains exactly $M$ points, here $i^2=-1$. So, the number of degenerate isosceles triangles is at least $M^3\cdot \frac{|E|}{M}=M^2|E|$, which is bigger than $|E|^3/q$ if $|E|\le q^{1/2}M$. For example, if $|E|\sim q^{4/3}$, one can take $M=q^{\frac{5}{6}+\epsilon}$ for any $\epsilon>0$.  The same happens for the case of Theorem \ref{th:primefields}. In other words, if one wishes to remove the condition $q\equiv 1\mod 4$, then the best hope with this approach is to show that the inequality (\ref{first-eq}) still holds when replace $D(E)$ by $D^*(E)$, where $D^*(E)$ is the set of triples $(x, y, z)\in D(E)$ with $||\overline{y}-\overline{z}||\ne 0$.
\end{remark}
\begin{remark}\label{rm4.2}
We note that by using a bisector line energy estimate due to Hanson, Lund, and Roche-Newton \cite[Theorem 3]{ben}, the proof of Theorem \ref{th:primefields} also implies Corollary \ref{cor1.3}. However, the method in \cite{ben} is very difficult to extend to higher dimensions. This explains why we need to employ techniques from Fourier extension/restriction theory to prove Theorem \ref{th:main}. 
\end{remark}
Remark \ref{rm4.2} leads us to the following question: 

{\bf Question:}
Is it possible to use results from Fourier extension/restriction theory to get a non-trivial result on bisector hyperplane energy in $\mathbb{F}_q^d$? 

\section{Acknowledgements}
T. Pham would like to thank to the VIASM for the hospitality and for the excellent working condition.

\end{document}